\documentclass[12pt,twoside]{amsart}
\usepackage{amssymb}
\usepackage{verbatim}
\usepackage{amsmath}
\usepackage{bm}
\usepackage{a4wide}
\usepackage[latin1]{inputenc}
\usepackage[T1]{fontenc}
\usepackage{times}
\usepackage{amssymb,latexsym}
\usepackage{enumerate}
\usepackage[stable]{footmisc}
\usepackage{color}
\usepackage[colorlinks,linkcolor=black,citecolor=black,urlcolor=black]{hyperref}

\newcommand{\R}{\ensuremath{\mathbb{R}}}
\newcommand{\D}{\ensuremath{\mathbb{D}}}
\newcommand{\C}{\ensuremath{\mathbb{C}}}

\newcommand{\bea}{\begin{eqnarray*}}
\newcommand{\eea}{\end{eqnarray*}}

\makeatletter
\newcommand{\sumprime}{\if@display\sideset{}{'}\sum%
            \else\sum'\fi}
\makeatother

\begin{document}

\numberwithin{equation}{section}

\newtheorem{theorem}{Theorem}[section]
\newtheorem{proposition}[theorem]{Proposition}
\newtheorem{conjecture}[theorem]{Conjecture}
\def\theconjecture{\unskip}
\newtheorem{corollary}[theorem]{Corollary}
\newtheorem{lemma}[theorem]{Lemma}
\newtheorem{observation}[theorem]{Observation}
\newtheorem{definition}{Definition}
\numberwithin{definition}{section} 
\newtheorem{remark}{Remark}
\def\theremark{\unskip}
\newtheorem{kl}{Key Lemma}
\def\thekl{\unskip}
\newtheorem{question}{Question}
\def\thequestion{\unskip}
\newtheorem{example}{Example}
\def\theexample{\unskip}
\newtheorem{problem}{Problem}

\thanks{The first author was supported by National Natural Science Foundation of China,  No. 12271101. The  third author was supported by Natural Science Foundation of Guangdong Province,  No.  2023A1515030017  and  the Fundamental Research Funds for the Central Universities,  Sun Yat-sen University}

\address [Bo-Yong Chen] {Department of Mathematical Sciences, Fudan University, Shanghai, 200433, China}
\email{boychen@fudan.edu.cn}

\address [John Erik Forn\ae ss]  { Department of Mathematical Sciences, Norwegian University of Science and Technology(NTNU), Sentralbygg 2, Alfred Getz vei 1, 7034 Trondheim, Norway} 
\email{john.fornass@ntnu.no}

\address  [Jujie Wu ]{School of Mathematics (Zhuhai), Sun Yat-Sen University, Zhuhai, Guangdong 519082, P. R.
 	China}
\email{wujj86@mail.sysu.edu.cn}

\title{Density in weighted Bergman spaces and Bergman completeness of Hartogs domains}
\author{Bo-Yong Chen, John Erik Forn\ae ss and Jujie Wu}

\date{}

\begin{abstract}
   We study the density of functions which are holomorphic in a neighbourhood of the closure $\overline{\Omega}$ of a bounded non-smooth pseudoconvex domain $\Omega$,  in the Bergman space $ H^2(\Omega ,\varphi)$ with a plurisubharmonic weight $\varphi$.  As an application,  we show that the Hartogs domain 
 $$
\Omega _\alpha : = \{(z,w) \in D\times \C:  |w|< \delta^\alpha_D(z)   \}, \ \ \  \alpha>0,
$$
where $D\subset \subset \C$ and $\delta_D$ denotes the boundary distance, is Bergman complete if and only if  every boundary point of $D$
 is non-isolated.

  \bigskip

  
  \smallskip
  
\end{abstract}

\maketitle

\tableofcontents

\section{Introduction}

Let $\Omega\subset\mathbb C^n$ be a bounded domain and $\varphi$ a measurable function in $\Omega$. We denote by $L^2(\Omega ,\varphi)$ the space of measurable functions $f$ in $\Omega $ such that   
$$
\|f\|_{L^2(\Omega, \varphi)} ^2 : = \int_{\Omega } |f|^2 e^{-\varphi} < + \infty.
$$
Here and in what follows,  we omit the Lebesgue measure $d\lambda$ in the integrals.
Let $$H^2(\Omega,\varphi) = \mathcal{O}(\Omega) \cap L^2(\Omega ,\varphi),$$   where $ \mathcal{O}(\Omega)$ is the set of holomorphic functions in $\Omega$.  It is of classical interest to consider the problem of finding conditions which ensure that $\mathcal O(\overline{\Omega}) \cap H^2(\Omega, \varphi)$ is dense in $H^2(\Omega, \varphi)$,  where $\mathcal O(\overline{\Omega})$ denotes the set of holomorphic functions in a neighbourhood of  $\overline{\Omega}$.  In case $n=1$ and $ \varphi =0$,  this problem has been investigated extensively (see \cite{Carleman}, \cite{Havin}, \cite{Hedberg});  for $\varphi \neq 0$,  one may refer to \cite{Taylor1971},  \cite{Hedberg1965},  \cite{BFW2020},  \cite{BFW2021}, etc.  However, for $n>1$, only a few results are known.  
 It is   known that the $L^2$-approximation does not hold when $\Omega$ is the Hartogs triangle (see e.g., \cite{Shaw2015}).  Note that the closure of the Hartogs triangle does not admit a Stein neighbourhood basis. It is then natural to ask the following.

\begin{problem} \label{problem1}
	Suppose $\overline{\Omega}$ admits a Stein neighbourhood basis and 
	$\varphi$ is a plurisubharmonic (psh) function in  $ \Omega $.  Under what conditions is it possible to conclude that  $\mathcal O(\overline{\Omega})\cap H^2(\Omega, \varphi)$ is dense in  $H^2(\Omega, \varphi)$? 
\end{problem}

The answer of Problem $1$ is affirmative when $\varphi=0$ if one imposes some regularity hypothesis on the boundary $\partial \Omega$, e.g., when $\partial \Omega$ can be described locally as the graph of a continuous function  (see \cite{Chen2000-1}).

In this paper, we first  show the following. 
\begin{theorem}\label{th:moregeneral}
 Let  $\Omega$ be a bounded pseudoconvex domain in $\C^n$ and $\varphi$  a psh function in a neighbourhood of $\overline{\Omega}$.   Suppose there exists a Stein neighbourhood basis  $\{\Omega^t\}_{0<t\le t_0}$ of\/  $\overline{\Omega}$\/ such that  
 $$
 \eta(t)\le  \delta_{\Omega^t}(z) : = d (z, \partial \Omega^t) \le  t,  \  \forall\,z\in \partial \Omega, 
$$  
where $\eta (t)$ is a non-negative continuous  increasing  function about $t$ satisfying
 \begin{eqnarray} \label{eq:condition}
\int_0^{r_0} \frac{dt}{t\log [t/\eta(t)]}=\infty,
\end{eqnarray}
for some $r_0\ll 1$.
Then $\mathcal O(\overline{\Omega}) \cap H^2(\Omega, \varphi) $ is dense in $H^2(\Omega, \varphi)$.
\end{theorem}

\begin{remark}
The special case $\varphi = 0$ and  $\eta(t) = Ct^\alpha$,  $C>0,  \alpha >1$  has been verified in  \cite{Chen99}.
 
\end{remark}

As an application, we deduce the following.
\begin{corollary} \label{co:co1}
	Suppose $\Omega \subset \C^n$ satisfies conditions of Theorem \ref{th:moregeneral} and $\overline{\Omega}$ is polynomially convex.  Let $\varphi$ be a psh function in a neighbourhood of $\overline{\Omega}$ with $1 \in H^2(\Omega, \varphi)$.   Then polynomials are dense in $H^2(\Omega, \varphi)$.	
\end{corollary}

If one poses certain additional assumption inside $\Omega$, then  the integral in $(\ref{eq:condition})$ can be finite.

\begin{theorem} \label{cor:Main22}
Let  $\Omega$ be a bounded pseudoconvex domain  in $\C^n$ and $\varphi$  a psh function in a neighbourhood of $\overline{\Omega}$.    Suppose
\begin{enumerate}  
	
	\item 	$\Omega$ admits a negative continuous psh function $\rho$   such that 
	$-\rho\lesssim  \delta_\Omega^\alpha$;
	
	\item  there exists a Stein neighbourhood basis  $\{\Omega^t\}_{0<t\le t_0}$ of\/  $\overline{\Omega}$\/ such that  
\begin{eqnarray}
	C e^{-C_1/t^\beta}  \le  \delta_{\Omega^t}(z)\le  t,  \  \forall\,z\in \partial \Omega, 
\end{eqnarray}
	 for some $\beta<\alpha/2$ and $C, C_1>0$.
	\end{enumerate}
Then $\mathcal O(\overline{\Omega})\cap H^2(\Omega, \varphi)$ is dense in $ H^2(\Omega, \varphi)$.
\end{theorem}

 Since every bounded simply-connected domain in $\mathbb C$ admits a negative continuous subharmonic function $\rho$ with  $-\rho\lesssim \delta_\Omega^{1/2}$ (see e.g., \cite{CarlesonGamelin}),   we can lighten the statement of Theorem \ref{cor:Main22}  as follows.

\begin{corollary}
	Let $\Omega$ be a bounded simply-connected domain in $\mathbb C$ and $\varphi$ a psh function in a neighbourhood of $\overline{\Omega}$.  If there exists a Stein neighbourhood basis  $\{\Omega^t\}_{0<t\le t_0}$ of\/  $\overline{\Omega}$\/ such that  
	$$ 
	C_1e^{-C_2/t^\beta}  \le  \delta_{\Omega^t}(z)\le  t,  \  \forall\,z\in \partial \Omega, 
	$$  
	 for some $\beta<1/4$ and $C_1, C_2>0$,
then $\mathcal O(\overline{\Omega}) \cap H^2(\Omega, \varphi) $ is dense in $ H^2(\Omega, \varphi)$.
 
	\end{corollary}

In the Appendix, we will construct a class of domains which satisfy Theorem  \ref{cor:Main22} but may  not satisfy  $\eta(t) = Ct^\alpha$, $C>0$, $\alpha >1$.

For the special weight $\varphi = - \alpha \log \delta$,  $\alpha >0$,  where $\delta$ is the boundary distance,  we have the following.

\begin{theorem} \label{th:deltabijin1}
	Let $D$ be a domain in $\mathbb{C}^n$. Suppose there exists a Stein neighbourhood basis  $\{D^t\}_{0<t\le t_0}$ of\/  $\overline{D}$\/ such that    
$$
\Lambda_t:=\sup_{z\in \partial D} \delta_{D^t}(z) \rightarrow 0, \quad  ( t\to  0 ), 
$$	
where $\delta_{D^t}$ (resp. $\delta_D$) is the Euclidean boundary distance of $D ^{t}$ (resp. $D$). Then  for  $\alpha > 0$, $\mathcal{O} \left(\overline{D}  \right)   \cap  H ^{2}   \left (D,   - \alpha \log \delta_D  \right)   $ is dense in $ H ^{2}   \left (D,   - \alpha \log \delta_D  \right)   $.
\end{theorem}

  As a consequence of Theorem \ref{th:deltabijin1},  one can easily obtain the following statement:
  
  \begin{corollary}\label{co:proofofco1}
  	Polynomials are dense in $ H ^{2}   \left (D,   - \alpha \log \delta_D  \right)   $ for all $\alpha > 0 $ if $D \subset \mathbb{C}^n$ is a bounded domain such that $\overline{D}$ is polynomially convex and $\overline{D} ^o =D$.
  \end{corollary}
 
In this paper,   as an important application of the $L^2$-approximation theorem,  we obtain the completeness of the Bergman metric for certain Hartogs domains.  

The celebrated  Kobayashi's criterion implies that if the Bergman kernel $K_\Omega $ of a domain $\Omega$ is an exhaustion function  and every $L^2$-holomorphic function in $\Omega$ can be approximated by bounded holomorphic functions in $\Omega$,  then the domain $\Omega$ is Bergman complete (i.e., the Bergman metric of $\Omega$ is complete).  It is known that  every bounded Bergman complete domain in $\C^n$ is  pseudoconvex (see \cite{Bremerm1955}). On the other hand,  the punctured disc is pseudoconvex but not Bergman complete.  Using Kobayashi's criterion, Bergman completeness and quantitative estimates of the Bergman distance have been obtained for various pseudoconvex domains (see \cite{Ohsawa1981}, \cite{D-O1995}),  \cite{JPZ2000},  \cite{Blocki2005},  \cite{NP2003},  \cite{NP2003-2},  \cite{PW2005}, etc.).  In particular, every hyperconvex domain is Bergman complete (see \cite{herbort-99},  \cite {Blocki-Pflug98}). 
 
Hartogs domains provide  various non-hyperconvex but Bergman complete examples (see \cite{PW2005}, \cite{Chen2013}),  e.g.,  it is shown \cite{Chen2013} that if $\varphi >0$ is a continuous psh function on a bounded pseudoconvex domain $D$ in $\C^n$ satisfying
\begin{eqnarray}\label{eq:assumption11}
\liminf_{z\rightarrow \partial  D} \frac{\varphi(z) }{ | \log \delta_D(z)|} = +\infty,
\end{eqnarray}
then the Hartogs domain 
$$
\Omega :  = \{ (z,w) \in D \times \C:  |w| < e^{-\varphi(z) }  \}
$$
 is Bergman complete.
  On the other hand,  if $D$ is the punctured disc $\D^*$ and $\varphi(z) \sim - N \log |z| $ as $z \rightarrow 0$, then $\Omega$ is not Bergman complete (compare \cite{Chen2013}). Thus,  (\ref{eq:assumption11}) cannot be weakened to 
  \begin{eqnarray}\label{eq:assumption111}
\liminf_{z\rightarrow \partial  D} \frac{\varphi(z) }{ | \log \delta_D(z)|} > 0,
\end{eqnarray}
e.g., when $\varphi = -\alpha \log \delta_D(z)$, $\alpha >0$.

 Here we shall prove
 
 \begin{theorem} \label{th:Bergmancomplete1}
Let $D\subset \subset \C$ be a domain and   $\varphi(z) = -\alpha \log \delta_D(z)$, $\alpha >0$.  Then the Hartogs domain
$$
\Omega _\alpha : = \{(z,w) \in D\times \C:  |w|<e^{-\varphi(z)}   \}
$$
is Bergman complete if and only if  every boundary point $z_0 $ of $D$
 is non-isolated, i.e.,  there exists a sequence  $\{z_k\}_{k=1}^\infty \subset \C\setminus D$ such that
$z_k \neq z_0$, $z_k \rightarrow z_0$ as $k\rightarrow \infty.$
\end{theorem}

This theorem surprised the authors: It seemed strange that the Bergman completeness, a complex analytic property,  can be characterized by a general topological condition non-existence of isolated boundary points.

More generally,  we are interested in under what conditions,  $\Omega$ is Bergman complete when $D$ is a pseudoconvex domain.  For high-dimensional cases,  we have the following.
 \begin{theorem}\label{th:approxidelta1}
Let $D$ be as in Theorem  \ref{th:deltabijin1}  locally and $\varphi(z) = -\alpha \log \delta_D(z)$,  $\alpha >0$.
 Then the Hartogs domain
$$
\Omega _\alpha : = \{(z,w) \in D\times \C:  |w|<e^{-\varphi(z)} =  (\delta_D(z)) ^\alpha   \}
$$
is Bergman complete. 
\end{theorem}

\section{Preliminaries}
A domain $\Omega \subset \C^n$  is called hyperconvex if there exists a negative continuous psh 
function $\rho$  on $\Omega$  such that $\{ \rho <c \} \subset \subset \Omega  $ for every $c < 0$.  Let $\Omega  \subset \C^n$ be a bounded hyperconvex domain. 
There are two concepts which measure the strength of hyperconvex. The first one is the so-called  Diederich-Forn\ae ss index (see, e.g.,  \cite{D-F1977},  \cite{D-F1977-2},  \cite{AdachiandB2015},  \cite{Fu-Shaw2016}, \cite{Harrington-2008}) defined as follows:
$$
\eta (\Omega): = \sup  \{ \eta \geq 0:  \exists \, \rho \in C(\Omega) \cap PSH^-(\Omega) ,  \  \text {s.t.} \    - \rho \thickapprox   \delta_\Omega ^\eta   \  \},
$$
 where $PSH^-(\Omega)$ ( \rm{resp.} $ SH^-(\Omega) )$ denotes the set of negative psh (\rm{resp.}   sh) functions on  $\Omega$. The second is the hyperconvex index (see \cite{Chen2017}) given by
 $$
\alpha(\Omega): = \sup  \{ \alpha \geq 0:  \exists \, \rho \in C(\Omega) \cap PSH^-(\Omega) ,  \  \text {s.t.} \  -\rho  \lesssim   \delta_\Omega  ^\alpha   \}.
$$

It is first shown in Diederich-Forn\ae ss \cite{D-F1977} that  $\eta(\Omega)>0$ if $\Omega$ is a bounded pseudoconvex domain with $C^2$-boundary. By Harrington \cite{Harrington-2008},  the boundary regularity can be relaxed to Lipschitz. Clearly, $\alpha(\Omega) \geq \eta(\Omega)$.  On the other hand, there are plenty of domains with wild irregular boundaries such that $\alpha(\Omega) >0$, while  it is difficult to verify  $\eta(\Omega) >0$.



The following theorems play  an important role in determining whether a domain is Bergman complete.
\begin{theorem}[Kobayashi criterion]  \cite{Kobayashi1959}\label{lm:kobayashi}
If there is a dense subset $S\subset H^2(\Omega, 0)$ such that for every $f\in S$ and for every
infinite sequence $\{  y_k\}_{k=1} ^\infty$ of points in $\Omega$ without accumulation points in $\Omega$, there is a subsequence $\{  y_{k_j}\}_{j=1} ^\infty$ such that
$$
\frac{| f( y_{k_j}) |^2}{ K_\Omega(  y_{k_j}) } \rightarrow 0 \ \ ( j \rightarrow \infty),
$$
then $\Omega$ is Bergman complete.
\end{theorem}

The following generalization of a theorem of Skwarczy$\rm{\acute{n}}$ski \cite{Skwarcz} plays an essential role in the proof of Theorem \ref{cor:Main22}.

\begin{theorem} \label{th:Sk1}
Let $\Omega$ be a bounded domain in $\mathbb C^n$ such that the boundary of $\Omega$ has Lebesgue measure zero and $\varphi$ a psh function  in a neighborhood of $\overline{\Omega}$.  Suppose there exists a family of domains $\{\Omega^t\}_{0<t\le t_0}$ such that $\overline{\Omega} \subset \Omega^t$ for all $t$  and $\Omega^t 
\searrow  \overline{\Omega}$ as $t \searrow 0$.   Let $K_{\Omega, \varphi}(z,w)$ and $K_{\Omega^t, \varphi}(z,w)$ denote the Bergman kernel of $H^2(\Omega, \varphi)$ and $H^2(\Omega^t, \varphi)$ respectively for $t \ll 1$.    Suppose 
	\begin{equation}\label{eq:BergConverg}
	K_{\Omega^t, \varphi} (w,w)\rightarrow K_{\Omega, \varphi} (w,w),\ \ \ \forall\, w\in E  \ \ \text{as}  \  \  t\rightarrow 0,
	\end{equation}
  where $E$ is  a set of uniqueness for $ H^2(\Omega, \varphi)$,  i.e., if $f\in H^2(\Omega, \varphi)$ satisfies $f=0$ on $E$, then $f=0$ on $\Omega$.   Then we have
	\begin{enumerate}
	\item   
 $  \bigcup _{0<t\ll 1} H^2(  \Omega^t, \varphi)  $  
	 is dense in $ H^2(\Omega, \varphi)$;
	\item  $
	K_{\Omega^t, \varphi} (z,w)\rightarrow K_{\Omega, \varphi} (z,w),\   \text{locally uniformly on }   \Omega \times \Omega, \ \ \text{as}  \  \  t\rightarrow 0.
$
	
	\end{enumerate}
\end{theorem}
\begin{proof}
 
We first show that the linear space spanned by the set $\{K_{\Omega,\varphi} (\cdot,w):w\in E\}$, say $H$, lies dense in $H^2(\Omega, \varphi)$. To see this,   take $f\in (\overline{H})^\bot$.  For every $w\in E$,  we have
$$
0=\int_\Omega f\cdot K_{\Omega,\varphi} (w,\cdot)e^{-\varphi(\cdot)}  =f(w),
$$
by the reproducing property. Thus $f=0$ on $\Omega$. It follows that for each $f\in H^2(\Omega, \varphi)$,  one can take a sequence 
$$
f_k=\sum_{l=1}^{ k}  c^{(k)}_{ l} K_{\Omega,\varphi} (\cdot,w^{(k)}_{ l}  ), \ \ \  c^{(k)}_{ l}  \in \mathbb C,\ w^{(k)}_{ l}  \in E, 
$$
such that $\|f_k-f\|_{L^2(\Omega, \varphi)}\rightarrow 0$ as $k \rightarrow \infty. $  On the other hand,  for fixed $k$,  let $0<t\ll 1$,   we take 
$$
\tilde{f}^t_k:=\sum_{l=1}^{k}  c^{(k)}_{ l}  K_{\Omega^{t}, \varphi}(\cdot,w^{(k)}_{ l} )\in 
H^2(  \Omega^t, \varphi).
$$
Since
	\begin{eqnarray*}
		&& \int_\Omega |K_{\Omega^t, \varphi}(\cdot, w^{(k)}_{ l}  )-K_{\Omega, \varphi}(\cdot,w^{(k)}_{ l}  )|^2e^{-\varphi(\cdot)}   \\
		& = & \int_\Omega |K_{\Omega^t, \varphi}(\cdot,w^{(k)}_{ l} )|^2 e^{-\varphi(\cdot)}   + \int_\Omega |K_{\Omega, \varphi}(\cdot,w^{(k)}_{ l} )|^2e^{-\varphi(\cdot)}   \\
		&& -2{\rm Re\,} \int_\Omega K_{\Omega^t, \varphi}(\cdot,w^{(k)}_{ l} ) \overline{K_{\Omega, \varphi}(\cdot,w^{(k)}_{ l} )} e^{-\varphi(\cdot)}  \\
		& = & \int_\Omega |K_{\Omega^t, \varphi}(\cdot,w^{(k)}_{ l}  )|^2 e^{-\varphi(\cdot)}   - 2K_{\Omega^t, \varphi}(w^{(k)}_{ l},   w^{(k)}_{ l} ) +K_{\Omega^t, \varphi}(w^{(k)}_{ l} ,  w^{(k)}_{ l} ) \\ 
		& \le & K_{\Omega, \varphi}(w^{(k)}_{ l} ,  w^{(k)}_{ l}  )-K_{\Omega^t, \varphi}(w^{(k)}_{ l}  ,  w^{(k)}_{ l}  )\\
		& \rightarrow & 0,    \ \ \text{as } \ \   t \rightarrow 0 , 
	\end{eqnarray*}
we get  $\|\tilde{f}^t_k -f_k \|_{L^2(\Omega, \varphi)}\rightarrow 0$ as $ t \rightarrow 0. $  
Hence,  we have $\|\tilde{f}^t_k-f\|_{L^2(\Omega, \varphi)} \rightarrow 0$ as $ t \rightarrow 0,  k \rightarrow \infty $.  This completes the proof of $(1)$.   

In order to prove  $K_{\Omega^t, \varphi} (z,w)\rightarrow K_{\Omega, \varphi} (z,w)$ locally uniformly,  it is sufficient to show that the point-wise convergence on the ``diagonal'' in   view of the Bergman inequality.
Since $\Omega^t \searrow  \overline{\Omega}$ as $t \searrow 0$,  we see that $K_{\Omega^t, \varphi} (z )$ is  increasing and bounded above by $ K_{\Omega, \varphi} (z)$,  for every $z\in \Omega$.  Thus the sequence $ \{ K_{\Omega^t, \varphi} (z)\} _t$ has a limit,  which we denote by $\widetilde{K}(z)$.  It is obvious that $\widetilde{K}(z) \leq K_{\Omega, \varphi} (z) $ for every $z\in \Omega$.  On the other hand,  for $f \in  H^2(\Omega^t, \varphi) $ with $t \ll 1$, we have 
\begin{eqnarray}\label{eq:kernel1}
|f(z)|^2 \leq K_{\Omega^t, \varphi} (z ) \| f\|^2_{ L^2(\Omega^t, \varphi) }.
\end{eqnarray}
  Since $ \bigcup _{0<t\ll 1} H^2(  \Omega^t, \varphi) $  is dense in $ H^2(\Omega, \varphi)$, we see that (\ref{eq:kernel1}) remains valid for every $f\in  H^2(\Omega, \varphi)$.   Letting $t\rightarrow 0$, we obtain
  \begin{eqnarray*} 
|f(z)|^2 \leq \widetilde{K}(z)  \| f\|^2_{  L^2(\Omega, \varphi) },
\end{eqnarray*} 
for $f\in  H^2(\Omega, \varphi)$.  Hence, $ K_{\Omega, \varphi} (z) \leq \widetilde{K}(z)$,  for every $z\in \Omega$.    The proof is complete.
  \end{proof}
  
\section{Proofs of Theorem  \ref{th:moregeneral} and Corollary \ref{co:co1}}
Let $\Omega \subset \C^n$ be a bounded domain. We define the relative extremal function of a (fixed) closed ball $\overline{B}\subset \Omega$ by 
\begin{eqnarray*}
\rho (z) : = \rho_{\overline{B}}(z) : = \sup \{ u(z): u \in PSH^-(\Omega), u_{\overline{B}} \leq -1  \}.
\end{eqnarray*}
  Clearly, the upper semi-continuous regularization $\rho^*$ of $\rho$ is psh on $\Omega$. In order to prove Theorem \ref{th:moregeneral}, we need the following very useful estimate for $\rho^*$.

\begin{proposition} [\cite{ChenHolder}]\label{prop:ChenHolder}
	Let $\Omega$ be a bounded pseudoconvex domain in $\C^n$. Set $\Omega_s = \{z \in \Omega: \delta (z) >s \} $ for $s>0$. Suppose there exist a number $0< \alpha <1$ and a family  of psh functions $\{ \psi _s \}_{0 < s \leq s_0}$  on $\Omega$ satisfying 
	\begin{eqnarray*}
		\sup _{\Omega} \psi _s < 1  \ \ \text{and} \ \ \inf _{\Omega \setminus \Omega_{\alpha s }}  \psi_s> \sup _{\partial \Omega_s} \psi_s,
	\end{eqnarray*}
	for all $s$. Then we have 
	\begin{eqnarray*}
		\sup_{\Omega \setminus \Omega_t} (- \rho ^ * ) \leq \exp \left [- \frac{1}{ \log {1/\alpha }} \int_{t/ \alpha}^{t_0}  
		\frac{\kappa _\alpha (s)}{s}ds \right ],
	\end{eqnarray*}	
	where $t /\alpha < t_0 \ll 1$ and
	\begin{eqnarray*}
		\kappa _ \alpha (s) :  = \frac{  \inf _{\Omega \setminus \Omega_{\alpha s }}  \psi_s  - \sup _{\partial \Omega _s} \psi_s} { 1 - \sup _{\partial \Omega_s} \psi_s } .
	\end{eqnarray*}
\end{proposition}
Another key observation to prove Theorem \ref{th:moregeneral} is as follows.  The proof is parallel to  Proposition 1.2 in \cite{BFW2020} ( or Theorem 4 in \cite{Chen2000-1}).
\begin{proposition} \label{le:volumeestimate}
	Let $\Omega \subset \C^n$ be a bounded domain and $\varphi$  a psh function  in a neighbourhood of $\overline{\Omega}$.  Suppose there exist a family of pseudoconvex domains $\{\Omega^t\}_{0<t\le t_0}$,  a sequence of continuous psh functions $\rho_t < 0$ defined on $\Omega^t$ for all $t$,   a sequence of  positive numbers $\{\varepsilon_t \}_t$ and a constant $C>1$ such that 
	\begin{enumerate}
		\item  $\Omega^t_{-\varepsilon_t} \subset \Omega$  where $\Omega^t_{-s} = \{ z \in \Omega^t : \rho _t(z) <- s \}  $  for positive number $s$,
		\item   $|  \Omega \setminus \Omega^t_{- C \varepsilon_t}	| \rightarrow 0$ \ \ \ $( t \rightarrow 0 )$.
	\end{enumerate}
	Then $\mathcal O(\overline{\Omega}) \cap H^2(\Omega, \varphi) $ is dense in $H^2(\Omega, \varphi)$.
	\end{proposition}

\begin{proof}  [Proof of Theorem \ref{th:moregeneral} ]         
	Let $\varrho_t$ (resp. $\varrho$) be the relative extremal function of a closed ball $\overline{B}\subset \Omega$ relative to $\Omega^t$ (resp. $\Omega$). Fix $0<\varepsilon <1$. For every $t$ the function 
	$$
	\psi_t :=\frac{\log 1/\delta_{\Omega^{\varepsilon t}}}{\log 2/\eta(\varepsilon t)}
	$$ 
	is psh  on $\Omega^{t'}$ and satisfies $\sup_{\Omega^{t'}}\psi_t<1$ for every $t'\ll \eta(\varepsilon t)$. Since    
	$$
	t\le  \delta_{\Omega^{\varepsilon t}}(z) \le t +\varepsilon t\ \ \ \text{for}\ \ \  \delta_{\Omega^{t'}}(z)=t,
	$$             
	we have
	$$
	\frac{\log 1/(t+\varepsilon t)}{\log 2/\eta(\varepsilon t)}\le \sup_{\delta_{\Omega^{t'}}=t} \psi_t \le \frac{\log 1/t}{\log 2/\eta(\varepsilon t)}.
	$$
	On the other hand, for $z\in \Omega^{t'}$ with $\delta_{\Omega^{t'}}(z)\le \alpha t$ we have
	$
	\delta_{\Omega^{\varepsilon t}}(z) \le \alpha t + \varepsilon t,
	$
	so that 
	$$
	\inf_{\delta_{\Omega^{t'}}\le \alpha t}\psi_t \ge \frac{\log 1/(\alpha t+\varepsilon t)}{\log 2/\eta(\varepsilon t)}.
	$$
	Fix $\alpha,\varepsilon$ with $\alpha+\varepsilon<1$. Then we have
	$$
	\kappa_\alpha(t) \ge \frac{\log 1/(\alpha+\varepsilon)}{\log [(2+2\varepsilon)t/\eta(\varepsilon t)]}.
	$$
	It follows from Proposition \ref{prop:ChenHolder}  that for a suitable positive constant $\tau$ depending only on $\alpha,\varepsilon$,
	$$
	\varepsilon_t:=    \sup_{\delta_{\Omega^{t'}}\le t} \left( -\varrho_{t'}^\ast \right)  \le  \exp\left(-\tau \int_{t/\alpha}^{ r_0}\frac{ds}{s\log [s/\eta(\varepsilon s)]}  \right)\\
	\rightarrow  0,
	$$ 
	as $t\rightarrow 0$.
	Thus we have  
	\begin{equation}\label{eq:1}
	\left\{z\in\Omega^{t'}: \varrho_{t'}^\ast(z)<-\varepsilon_t\right\}\subset \left\{z\in \Omega^{t'}: \delta_{\Omega^{t'}}(z)\ge t\right\}\subset\subset \Omega,
	\end{equation}
	provided $t'$ is sufficiently small.
	On the other hand, since $\varrho^\ast(z)\ge \varrho_{t'}^\ast(z)$ for all $z\in \Omega$, 
	we have
	$$
	\left\{z\in \Omega: \varrho_{t'}^\ast(z)\ge- 2 \varepsilon_t\right\}\subset \left\{z\in \Omega: \varrho^\ast(z)\ge -2\varepsilon_t\right\}.
	$$
	By Proposition 5.1 in \cite{ChenHolder}, we know that $\varrho=\varrho^\ast$ is a negative continuous psh exhaustion function on $\Omega$. Thus we obtain
	\begin{equation}\label{eq:2}
	\left| \left\{z\in \Omega: \varrho_{t'}^\ast(z)\ge- 2 \varepsilon_t\right\}\right| \rightarrow 0 ,     
	\end{equation}
	as $t\rightarrow 0$. The assertion follows from Proposition \ref{le:volumeestimate}.
\end{proof}

\begin{proof}[Proof of Corollary \ref{co:co1}]
	Suppose $\{\Omega^t\}_{0<t\le t_0}$ is a Stein neighbourhood basis of $\overline{\Omega}$. By Theorem \ref{th:moregeneral}, for every $f\in H^2(\Omega, \varphi)$ and every $\varepsilon >0$ there exists $F\in \mathcal{O}(\Omega^t)$ so that 
	\begin{eqnarray*}
		\int_{\Omega}|F-f|^2  e^{-\varphi}   < \varepsilon.
	\end{eqnarray*}	
	Since $\overline {\Omega}$ is polynomially convex,   $F$ can be approximated uniformly by polynomials on the compact set $\overline{\Omega}$ by the theorem of Oka-Weil. Thus,  we have a sequence of polynomials $P_k$ such that 
	$$
	\int_\Omega | P_k - F|^2 e^{-\varphi}   \leq  \sup_{\overline{\Omega}} | P_k - F|^2\int _\Omega e^{-\varphi}   < \varepsilon \int _\Omega e^{-\varphi}  = : \varepsilon M,
	$$
	when $k$ is very large, 
	which implies that we can find a sequence of polynomials $P_k$ such that 
	$$
	\int _\Omega |f - P_k |^2 e^{-\varphi }    < C \varepsilon,
	$$
	for sufficiently large $k$.
	The proof is complete.	
\end{proof}

\section{Proofs of  Theorem \ref{cor:Main22} and Theorem \ref{th:deltabijin1}}

In order to prove Theorem  \ref{cor:Main22},  we will first give a  $L^2$-approximation theorem as follows.   
\begin{theorem}\label{th:Main2}
	Let $\Omega$ be a bounded pseudoconvex domain in $\mathbb C^n$ and $\varphi$  a psh  function in a neighbourhood of $\overline{\Omega}$. Let $E$ be a compact set of uniqueness for $H^2(\Omega, \varphi)$ (e.g., $E$ is a closed ball in $\Omega$).  Set
	$$
	\nu_E(t):=\sup_{w\in E} \int_{\delta_\Omega<t} |K_{\Omega, \varphi}(\cdot,w)|^2 e^{-\varphi(\cdot)} ,\ \ \ t>0.
	$$
	Suppose  there exists a Stein neighbourhood basis  $\{\Omega^t\}_{0<t\le t_0\ll 1}$ of\/  $\overline{\Omega}$\/ such that    
	$$
	\eta(t)\le  \delta_{\Omega^t}(z)\le  t,\ \ \ \forall\,z\in \partial \Omega,
	$$ 
	where $\eta$ satisfies
	\begin{eqnarray*}
		\sqrt {\nu_E(2 t)} \cdot \log (2t /{\eta(t)}) \rightarrow 0,
	\end{eqnarray*}
as $t\rightarrow 0$.
	Then $\mathcal O(\overline{\Omega})\cap H^2(\Omega, \varphi)$ is dense in $H^2(\Omega, \varphi)$.
\end{theorem}

\begin{proof}[Proof of Theorem \ref{th:Main2}]
	Set 
	$$
	\rho_t=-\log (2\delta_{\Omega^t}/\eta(t)) ,\ \ \ \Omega^{t'}=\{z\in \Omega^t: \delta_{\Omega^t}>\eta(t)/2\}.
	$$
	By Oka's lemma, we see that $\Omega^{t'}$ is a pseudoconvex domain  and $\rho_t\in PSH^-(\Omega^{t'})$. Moreover,  $\overline{\Omega}\subset \Omega^{t'}$.  We  choose two sequences of positive numbers as follows
	$$
	a_t=\log (2 t / \eta(t)),\ \ \  b_t=(1+\varepsilon_t)a_t,
	$$
	where $\varepsilon_t\rightarrow 0+$ will be determined later.  Let $\chi_t:\mathbb R\rightarrow [0,1]$ be a smooth cut-off function satisfying 
	$\chi_t=1$ on $(-\infty,-\log b_t]$, $\chi_t=0$ on $[-\log a_t,\infty)$, and
	$$
	|\chi_t'|\le \frac1{\log b_t-\log a_t}=\frac1{\log(1+\varepsilon_t)}\le \frac2{\varepsilon_t},
	$$
	for $t$ sufficiently small.  Fix $w\in E$ for a moment. Since $\{\delta_{\Omega^t}(z)>t\}\subset \Omega$, it follows  that 
	$$
	\chi_t(-\log(-\rho_t))\, K_{\Omega, \varphi} (\cdot,w)
	$$
	is a smooth function on $\Omega^{t'}$. 
	Let $u_t$ be the $L^2(\Omega^{t'}, \varphi)$ minimal solution of 
	$$
	\bar{\partial} u = K_{\Omega, \varphi} (\cdot,w) \bar{\partial} \chi_t(-\log(-\rho_t)) =: v_t.
	$$
	By the classical Donnelly-Fefferman estimate ( see \cite{B1996}, Theorem 3.1; or \cite{Blocki2004}),  
	 we have
	\begin{eqnarray*}
		\int_{\Omega^{t'}} |u_t|^2 e^{-\varphi}  & \le & C_0 \int_{\Omega^{t'}} |\bar{\partial} \chi_t(\cdot)|^2_{-i\partial\bar{\partial} \log (-\rho_t)} 
		|K_{\Omega, \varphi}(\cdot,w)|^2 e^{-\varphi(\cdot)} \\
		& \le & \frac{C_0}{\varepsilon_t^2} \,  \int_{a_t\le-\rho_t\le b_t}  |K_{\Omega, \varphi}(\cdot,w)|^2 e^{-\varphi(\cdot)} ,
	\end{eqnarray*}
	where $C_0$ is a universal constant and $|\bar{\partial} \chi_t(\cdot)|^2_{-i\partial\bar{\partial} \log (-\rho_t)}$ should be understood as the infimum of nonnegative locally bounded functions $H$ satisfying $i \bar{\partial} \chi_t  \wedge \partial  \chi_t  \leq H i \partial\bar{\partial} (-\log (-\rho_t))$ as currents.  Note that 
$$
a_t\le -\rho_t\le b_t \iff t \le \delta_{\Omega^t}\le \frac{\eta(t)}2\, e^{b_t}=:d_t.
$$
	Since $\delta_{\Omega^t}\ge \delta_\Omega$, we have
	$$
	\{a_t\le-\rho_t\le b_t\}\subset \{\delta_\Omega\le d_t\},
	$$
	so that 
	$$
	\int_{\Omega^{t'}} |u_t|^2 e^{-\varphi}  \le \frac{C_0}{\varepsilon_t^2} \,  \nu_w(d_t),
	$$
	where 
	$$
	\nu_w(t)=\int_{\delta_\Omega<t} |K_{\Omega, \varphi}(\cdot,w)|^2e^{-\varphi(\cdot)} .
	$$
	We require that $d_t\rightarrow 0$.
	As 
	$
	\{\delta_\Omega > d_t\}$ is contained in $\{-\rho_t>b_t\}$, where $v_t=0$, i.e., $u_t$ is holomorphic, the mean-value inequality implies 
	$$
	|u_t(w)|^2 \le C \int_\Omega |u_t|^2\le C 
	\int_\Omega |u_t|^2 e^{-\varphi}  \le  
	\frac{C}{\varepsilon_t^2}  \, \nu_w(d_t),
	$$ 
	where $C$ is a  constant depending only on $E,\Omega$. It follows that
	$$
	f_t:=\chi_t(-\log(-\rho_t))K_{\Omega, \varphi} (\cdot,w)-u_t
	$$
	is  holomorphic on $\Omega^{t'}$,  and since $u_t\bot H^2(\Omega^{t'}, \varphi)$ and $f_t\in H^2(\Omega^{t'}, \varphi)$,  we have
	$$
	\chi_t(-\log(-\rho_t))K_{\Omega, \varphi}(\cdot,w)= f_t\oplus u_t. 
	$$
	Thus
	\begin{eqnarray*}
		\|f_t\|_{L^2(\Omega^{t'}, \varphi)} & \le &   \|\chi_t(-\log(-\rho_t))K_{\Omega, \varphi} (\cdot,w)\|_{L^2(\Omega^{t'}, \varphi)}   \\
		& \le &    \| K_{\Omega, \varphi}(\cdot,w)\|_{L^2(\Omega, \varphi)} =K_{\Omega, \varphi}(w,w)^{1/2}.                     
	\end{eqnarray*}
Moreover, 
	$$
	|f_t(w)|\ge K_{\Omega, \varphi}(w,w)-|u_t(w)|\ge K_{\Omega, \varphi}(w,w)- \frac{C}{\varepsilon_t}\, \sqrt{\nu_w (d_t)}.
	$$
	Thus

	\begin{eqnarray*}
	K_{\Omega^{t'}, \varphi}(w,w)\ge \frac{|f_t(w)|^2}{\|f_t\|_{L^2(\Omega^{t'}, \varphi)}^2} &\ge &\left(  \frac{K_{\Omega, \varphi}(w,w)-\frac{ C}{\varepsilon_t}\,  \sqrt{\nu_w (d_t)} }{ K_{\Omega, \varphi}(w,w)^{1/2}  } \right) ^2 \\ \nonumber 
	&\ge &  K_{\Omega, \varphi}(w,w)- \frac{2C}{\varepsilon_t}\, \sqrt{\nu_w (d_t)}.
\end{eqnarray*}
Since $ K_{\Omega^{t'}, \varphi}(w,w)\le  K_{\Omega, \varphi}(w,w)$, we conclude that $K_{\Omega^{t'}, \varphi}(w,w)\rightarrow K_{\Omega, \varphi}(w,w)$ for every $w\in E$.  Consequently, $\mathcal O(\overline{\Omega})\cap H^2(\Omega, \varphi )$ is dense in $H^2(\Omega, \varphi)$,  in view of Theorem \ref{th:Sk1},  provided 
	\begin{equation}\label{eq:converge}
	\frac{\nu_w(d_t)}{\varepsilon_t^2} \rightarrow 0,\ \ \ \forall\,w\in E,   \ \ \ \text{as} \ \  d_t \rightarrow  0.
	\end{equation}
	Let us choose a sequence of positive numbers $\{\tau_t\}$ such that $\tau_t \rightarrow \infty$ and 
	$$
	\tau_t \sqrt{\nu_E(2t)}\,  \log(2t/\eta(t))\rightarrow 0.
	$$
	Clearly,    (\ref{eq:converge}) holds whenever 
	\begin{equation*}\label{eq:epsilon}
	\varepsilon_t=\tau_t \sqrt{\nu_w (d_t)}.
	\end{equation*}
	Since $d_t=\frac{\eta(t) }2e^{b_t}$, it follows that $d_t$ has to verify the following
	\begin{equation}\label{eq:key}
	\log \frac{d_t}{t} = \tau_t \sqrt{\nu_w (d_t)}\, \log (2t/\eta(t)).
	\end{equation}
	In order to solve the above equation, we consider the following continuous function
	$$
	\gamma(x):=\log \frac{x}{t}-\tau_t \sqrt{\nu_w (x)}\, \log (2 t/\eta(t)),\ \ \ x>0.
	$$
	Since $\gamma(0+)=-\infty$ and $\gamma(2t)>0$ for all $t$ sufficiently small, there exists $0<d_t<2t$ such that $\gamma(d_t)=0$, i.e., (\ref{eq:key}) holds. 
\end{proof}

We also need the following $L^2$-boundary decay estimate  for the weighted Bergman kernel (see \cite{Chen2017} for the standard unweighted case).

\begin{proposition} \label{prop:estimatedecay}
	Let $\Omega \subset \C^n$ be a pseudoconvex domain and  $\rho$ a negative continuous psh function on $\Omega$.   Let $\rm{S}\subset \subset \Omega$ be a subdomain and $\varphi$ a psh function on $\Omega$. Then for every $0< r<1$, there exists a constant C which depends on $S$, $r$ such that
	\begin{eqnarray*}
		\int _{-\rho < \varepsilon} |K_{\Omega, \varphi} (\cdot, w)|^2 e^{-\varphi(\cdot)}   \leq C\varepsilon^r,
	\end{eqnarray*}
for all $w\in S$. 
\end{proposition}
\begin{proof}
	Put $\psi = -r \log (-\rho)$. Then we have 
	\begin{eqnarray*}
			ri\partial\overline{\partial} \psi \geq i\partial \psi \wedge \overline{\partial} \psi.
	\end{eqnarray*} 
By Theorem 2.1 in \cite{BC2000},  we have  
\begin{eqnarray}\label{ineq:1}
	\int_\Omega |P_\varphi (f)|^2 e^{\psi -\varphi}  \leq C\int_\Omega |f|^2 e^{\psi -\varphi} ,
\end{eqnarray}
for every $f \in L^2(\Omega, \varphi - \psi)$,
where $P_\varphi (f)(z):= \int_{\Omega} f(\zeta) K_{\Omega, \varphi} (z, \zeta) e^{-\varphi(\zeta)} $ is the Bergman projection of $f$. Put $f= \chi_S K_{S, \varphi}(\cdot, w)$, $w\in S$, where $\chi_S$ is the characteristic function of $S$. Then 
\begin{eqnarray*}
P_\varphi (f) (z)	&=& \int_{\Omega} \chi_S(\zeta) K_{S, \varphi}(\zeta, w) K_{\Omega, \varphi} (z, \zeta) e^{-\varphi(\zeta)}   \nonumber \\
& = & \int_{S}  K_{S, \varphi} (\zeta, w) K_{\Omega, \varphi}(z, \zeta) e^{-\varphi(\zeta)}   \nonumber \\
& = & K_{\Omega, \varphi}(z,w),
\end{eqnarray*}
for all $w\in S$ and 
\begin{eqnarray*}
\int_{-\rho < \varepsilon }|K_{\Omega, \varphi} (\cdot,w)|^2 e^{ -\varphi} 
&\leq & \varepsilon^{r} \int_{-\rho < \varepsilon }|K_{\Omega, \varphi} (\cdot,w)|^2 |
\rho|^{-r} e^{ -\varphi} \nonumber \\
&\leq & \varepsilon^{r} \int_{\Omega }|K_{\Omega, \varphi}(\cdot,w)|^2 e^{ \psi-\varphi}    \nonumber \\
&\leq & C \varepsilon^{r} \int_{\Omega }|\chi_S K_{S, \varphi}(\cdot, w)|^2 e^{ \psi-\varphi}  \ \ (\text{by \ref{ineq:1}} )\nonumber \\
&= & C_{S,r}\varepsilon^{r} .  
\end{eqnarray*}
	
\end{proof}


\begin{proof}[Proof of Theorem \ref{cor:Main22}]
	By Proposition \ref{prop:estimatedecay} there exists,  for every $\alpha'<\alpha$ and every compact subset $E\subset \Omega$ (with nonempty interior), a constant $C_3>0$ such that 
	$$
	\nu_E (t) \le C_3 t^{\alpha'}.
	$$  
	Note that {\tiny }
	$$
	\sqrt {\nu_E(2 t)} \cdot \log (2t/{\eta(t)})\le C_3^{1/2}(2t)^{\alpha'/2}\left(\log (2t)- \log C_1+C_2/t^\beta\right)\rightarrow 0,
	$$
	provided $\alpha'>2\beta$.  By using Theorem \ref{th:Main2}  we complete the proof.
\end{proof}

 \begin{proof} [Proof of Theorem \ref{th:deltabijin1} ] 
Replacing $\delta_D$ by $c\delta_D$ with $c>0$ sufficiently small, we may always assume that $\delta_D$ (as well as $\delta_{D^t}$) is as small as  wanted. We define the following functions on $D^t$:
$$
\psi_t=-\frac{1}{2}\log {\log {1/\delta _{D^t} }},\quad \varphi_{t}=\alpha \log {1/\delta _{D^t}}  -\frac{1}{2} \log {\log {1/\delta _{D^t} }}.
$$ 
We have 
\begin{align*}
	i\partial\bar{\partial}\left(\varphi_t+\psi_t \right)=& \alpha i\partial \bar{\partial}\log {1/\delta_{D^t} } -i \partial   \bar{\partial} \log \log {1/\delta _{D^t} } \\
= &\left  (\alpha - \frac{1}{ \log {1/\delta_{D^t}} } \right )   i\partial \bar{\partial}\log {1/\delta _{D^t} }  + \frac{i \partial  \log {1/\delta _{D^t} }  \wedge  \bar{\partial}   \log {1/\delta_{D^t} }}{ (\log {1/\delta _{D^t} })^2 }\\
\geq & \frac{i \partial  \log {1/\delta _{D^t}}  \wedge  \bar{\partial}   \log {1/\delta _{D^t} }}{ (\log {1/\delta _{D^t} })^2 }\\
\geq & 2i\partial\psi_t \wedge\bar{\partial}\psi_t 
\end{align*}
in the sense of distribution since $\log_{}{1/\delta_{D^t} }$ is psh on $D^t$ by Oka's lemma. 
Furthermore, 
\begin{align*}
	i\partial\bar{\partial} \varphi_t  =& \left  (\alpha -  \frac 12 \frac{1}{ \log {1/\delta_{D^t}} } \right )   i\partial \bar{\partial}\log {1/\delta _{D^t}}  +   \frac 12 \frac{i \partial  \log {1/\delta _{D^t} }  \wedge  \bar{\partial}   \log {1/\delta _{D^t} }}{ (\log {1/\delta _{D^t} })^2 }\\
\geq & \frac 12    \frac{ i \partial  \log {1/\delta _{D^t} }  \wedge  \bar{\partial}   \log {1/\delta _{D^t}}}{ (\log {1/\delta _{D^t}})^2 }\\
\geq & 0.
\end{align*}
Thus  $\varphi_t $ is psh on $D^t$.

Let $0\le \chi \le1$ be a $\mathit{C}^{\infty}$ function on $\mathbb{R}$ such that $\chi =0$ on $\left(0,+\infty\right)$ and $\chi=1$ on $\left(-\infty,-\log_{}{2}\right)$. For each $f\in  H ^{2}   \left (D,   - \alpha \log \delta_D  \right) $ and every $\varepsilon >0$, we define 
$$
v _{t,\varepsilon }: = f\bar{\partial}\chi \left ( \log {\log {1/\delta _{D^t} } }-\log {\log {1/\varepsilon } } \right ).
$$
We shall use a Donnelly-Fefferman type $L^2$- estimate for the $\bar{\partial}$- equation, following an idea of Berndtsson-Charpentier \cite{BC2000}.   Note first that $v _{t,\varepsilon }$ is well-defined on $D^t$ for sufficiently small $t$ (e.g.,$\Lambda_{t} : = \sup _{z \in \partial D }\delta_{D^t} (z) <\varepsilon $). Take a bounded pseudoconvex domain $ \widetilde{D^t }  $ such that $\overline{D}\subset  \widetilde{D^t }   \subset \subset D^t$. Since $\varphi _t$ is a bounded (psh) function on $\widetilde{D^t }   $, $ v_{t,\varepsilon }\in L_{\left ( 0,1 \right ) }^{2}\left ( \widetilde{D^t },   \varphi _t \right)$ (Note that $\delta_{D^t}$ is a Lipschitz  function). Thus by H\"omander's theorem, there is a minimal solution $u_{t,\varepsilon }$ of the equation $\bar{\partial }u=v_{t,\varepsilon }$ in $L^{2}\left(\widetilde{D^t },   \varphi _{t}  \right)$.  Since $u_{t,\varepsilon }e^{\psi _t}\perp$ Ker $\bar{\partial}$ in $L^{2} \left( \widetilde{D^t } ,   \varphi _t+ \psi_t \right)$,  we have
\begin{align*}
	\int_{\widetilde{D^t } } \left | u_{t,\varepsilon }\right |^2e^{\psi _{t}-\varphi _{t}  } \le &\int_{\widetilde{D^t }} \left | \bar{\partial }\left (u_{t,\varepsilon }e^{\psi _{t} }\right )\right |_{i\partial \bar{\partial}\left(\varphi_{t}+\psi _{t} \right )}^{2}e^{-\varphi _{t}-\psi _{t}  }     \\
	=&\int_{\widetilde{D^t } } \left | v_{t,\varepsilon }+\bar{\partial}\psi _{t}\wedge u_{t,\varepsilon } \right |_{i\partial \bar{\partial}\left(\varphi_{t}+\psi _{t} \right )}^{2}e^{\psi _{t}-\varphi _t  } \\
	\le & \left ( 1+\frac{1}{r}\right ) \int_{\widetilde{D^t }} \left | v_{t,\varepsilon } \right |_{i\partial \bar{\partial}\left(\varphi_{t}+\psi _{t} \right )}^{2}e^{\psi _{t}-\varphi _{t}  }\\
	&+\left ( 1+r \right )\int_{\widetilde{D^t }} \left | \bar{\partial}\psi _{t}\right |_{i\partial \bar{\partial}\left(\varphi_{t}+\psi _{t} \right )}^{2}\left | u_{t,\varepsilon }  \right |^2 e^{\psi _{t}-\varphi _{t} }, 
\end{align*}
where $r>0$. We may choose $r$ sufficiently small such that the second term in the last inequality may be absorbed by the left side and we obtain 
\begin{align*}
	\int_{\widetilde{D^t } }\left | u_{t,\varepsilon } \right |^{2}e^{\psi _{t}-\varphi _{t}  } \le& C_{\alpha  }\int_{\widetilde{D^t } }\left | v_{t,\varepsilon } \right |_{i\partial \bar{\partial}  \left ( \varphi _{t}+\psi _{t} \right ) }^{2} e^{\psi _{t}-\varphi _{t}} \\
	\le & C_{\alpha   }\int_{\left \{ \varepsilon \le \delta _{t}\le \sqrt{\varepsilon } \right \}   }^{}\left | f\right  |^{2} \delta_{D^t  }^{\alpha }   \\
	\le & C_{\alpha   }\int_{\left \{ \frac12\varepsilon \le \delta \le \sqrt{\varepsilon } \right \}   }^{}\left | f\right  |^{2} \delta_{D^t }^{\alpha }\\
	\le & C_{\alpha  }\int_{\left \{ \frac 12\varepsilon \le \delta \le \sqrt{\varepsilon } \right \}   }^{}\left | f\right  |^{2} \delta_D^{\alpha }\\
\end{align*}
for all $t \leq t_{\varepsilon } \ll 1$, since $\delta_D \le \delta _{D^t} \le \delta_D +\Lambda _{t}$ on $\overline{D}$. Here $C_{\alpha  }>0$ is a generic constant depending only on $\alpha$. Put
$$
f_{t,\varepsilon}:=  f\chi \left ( \log_{}{\log_{}{1/\delta_{D^t}} }-\log_{}{\log_{}{1/\varepsilon } }  \right )-u_{t, \varepsilon }.
$$
Then $f_{t,\varepsilon}\in \mathcal{O} \left(\widetilde{D^t } \right)$ and 
\begin{align*}
	  \int_{D}^{}\left |f_{t, \varepsilon}- f\right  |^{2} \delta_D^{\alpha }   & \le    2\int_{\left \{ \varepsilon \le \delta \le \sqrt{\varepsilon } \right \}   }^{}\left | f\right  |^{2} \delta_D^{\alpha } +2\int_{\Omega}^{}\left |u_{t,\varepsilon}\right |^{2} \delta_{D^t}^{\alpha }  \\
	& \le   C_{\alpha  }\int_{\left \{ \frac 12\varepsilon \le \delta \le \sqrt{\varepsilon } \right \}   }^{}\left | f\right  |^{2} \delta_D^{\alpha } ,
\end{align*}
provided $t \leq t_{\varepsilon } \ll 1$. Thus we obtain the desired approximation of $f$.  
\end{proof}

 \begin{remark}
 Theorem \ref{th:deltabijin1} still holds if we suppose there is a sequence of pseudoconvex open sets $\{ D_j \} _j$ with $\overline{D} \subset D_j$ for all $j$ and $\sup _{ z\in \partial D} d(z, \partial D_j )  \rightarrow 0$.
 
 \end{remark}

\begin{proof} [  Proof \,of\,  Corollary \ref{co:proofofco1}]
 Since $\overline{D}$ is polynomially convex, it is well-known that there is,   for each neighbourhood $U$ of $D$,  a continuous psh exhaustion function $\rho$ in $\mathbb{C}^n$ so that
$$
\overline{D}\subset D_{0}:= \left \{ z\in \mathbb{C}^n: \rho \left ( z \right )<0  \right \}\subset U    
$$
( see \cite{Hormander1965} ). Since $D _{0}$ is  psedoconvex and $ \overline{ D}^o =D$, we may construct a sequence of pseudoconvex domains $D_j$ satisfying $\overline{D} \subset D_j$ for all $j$ and $\sup _{ z\in \partial D} d(z, \partial D_j )  \rightarrow 0$. Thus for every $\varepsilon>0$ and $f \in   H ^{2}   \left (D,   - \alpha \log \delta_D  \right)  $,  there is a  function $g$ which is holomorphic in a neighborhood of $D$ such that $\left \| g-f \right \|_{  L ^{2}   \left (D,   - \alpha \log \delta_D  \right)  }  < \varepsilon/2$.  On the other hand,  $g$ can be uniformly approximated by polynomials on the compact set $\overline{D}$ in view of Oka-Weil theorem.  Thus,  there is a polynomial $p\left(z\right)$ such that $\left \| p-g \right \|_{ L ^{2}   \left (D,   - \alpha \log \delta_D  \right) }  < \varepsilon/2$,  which implies that $\left \| p-f \right \|_{ L ^{2}   \left (D,   - \alpha \log \delta_D  \right)  }  < \varepsilon$. 
\end{proof}

\section{Proofs of Theorem \ref{th:Bergmancomplete1} and  Theorem \ref{th:approxidelta1}}

In what follows,  we assume that $D$ be a domain in $\C^n$ and $\varphi$  a continuous psh function on $D$.  Consider the Hartogs domain $\Omega$ in $\C^{n+1}$ as follows
$$
\Omega  = \{(z,w) \in D\times \C:  |w|<e^{-\varphi(z)} \}.
$$
The proof of theorem \ref{th:Bergmancomplete1} is based on several lemmas given below.
\begin{lemma}[Ligocka, \cite{Ligocka1989}] \label{prop:ligo}  
  $(1)$ Let $f\in H^2(\Omega, 0)$.  Then 
$ 
f(z,w) = \sum_{j=0} ^ \infty f_j(z) w^j , 
 $
where for every $j$,  $f_j \in H^2(D, 2(j+1) \varphi)$.
Moreover,
$$
\| f(z,w) \|^2 _{ L^2(\Omega, 0)} =  \sum_{j=0} ^ \infty \frac{ \pi} { j+1}  \|  f_j \|^2_ {L^2 (D, 2(j+1) \varphi)  }.
$$
 $(2)$  For  $z,t\in D,   \ \ w, s \in \mathbb{C} , \ \ |w|<e^{-\varphi(z)}, |s|<e^{-\varphi(t) }$,  we have 
\begin{eqnarray*}
K_{\Omega}( (z,w), (t,s)) = \sum\limits_{j=0}^{\infty}  \frac{j+1 }{\pi}  K_{D,  2(j+1)\varphi }(z,t) \langle w, s\rangle^j, 
\end{eqnarray*}
 where $ K_{D,  2(j+1) \varphi  }(z,t)$ is the weighted Bergman kernel of $H^2(D, 2(j+1) \varphi )$ and $K_{\Omega}((z,w), (t,s)) $ is the classical Bergman kernel of $H^2(\Omega, 0)$.
 
\end{lemma}

\begin{lemma} \label{lm:exhaustion}
Let $D$ be a bounded pseudoconvex domain in $\C^n$ and  $\varphi >0$  a continuous psh function on $D$ satisfying
\begin{eqnarray*}\label{eq:assumption1}
\lim _{z\rightarrow \partial D} \varphi(z)  = +\infty.
\end{eqnarray*}
Then the Hartogs domain $\Omega$  is Bergman exhaustive.
\end{lemma}

\begin{proof}
Let $\rho(z,w)  = \log |w| + \varphi(z)$.  Then it is psh on $\Omega$ and $\rho(z,w) = 0$ on $\partial \Omega \cap \{ w \neq 0\}  $, which means that $\Omega$ is locally hyperconvex at $\partial \Omega \cap \{ w \neq 0\}  $.
Since every hyperconvex domain is Bergman exhaustive (see \cite{Ohsawa1993}),   it follows from the well-known localization principle (see \cite{Ohsawa1981-1} or \cite{D-F-H1982} ) of the Bergman kernel that  $K_\Omega \rightarrow + \infty$ at  $\partial \Omega \cap \{ w \neq 0\}$.
On the other hand, by $(2)$ of the Ligocka formula in Lemma \ref{prop:ligo}, we know that
\begin{eqnarray}
  K_\Omega ((z,w), (z,w)) \geq   \frac{1 }{\pi} K_{D,  2 \varphi } (z) .
\end{eqnarray}
We claim that 
$$
 \frac{1 }{\pi} K_{D,  2 \varphi } (z)  \geq C e^{2\varphi(z)}  ,
$$
for some positive number $C$.
Take $f\in H^2 (D, 2\varphi)$.  Applying the Ohsawa-Takegoshi extension theorem for a single point $z$ in a similar way as Demailly \cite{Demailly1992}, we obtain a holomorphic function $F$ on $D$ with the following properties:   $F(z) = f(z)$, and 
$$
\int _D |F|^2e^{-2 \varphi}   \leq C |f(z)|^2e^{-2 \varphi(z)} .
$$
Hence 
$$
e^{2 \varphi(z) } \leq  C \frac{ |f(z)|^2 } { \int _D |F|^2e^{-2 \varphi}   }  = C \frac{ |F(z)|^2 } { \int _D |F|^2e^{-2 \varphi}   } \leq C K_{D,  2 \varphi } (z).
$$
Thus when $(z, w) \rightarrow \partial \Omega$, we have $z\rightarrow \partial D$ and  
$$  K_\Omega ((z,w), (z,w))  \geq C e^{2 \varphi(z) } \rightarrow  +  \infty.$$ This completes the proof.
\end{proof}

\begin{lemma} \label{lm:approxi}
Let $D\subset \subset \C$ be a domain whose boundary has no isolated points.   Let $\varphi(z) = - \alpha \log \delta (z)$ for $\alpha >0$.
Then for every $z_0\in \partial D$,  $H^2(\Omega,0) \cap \mathcal{O}_{(z_0,0)}$ is dense in $H^2(\Omega,0)$, where $ \mathcal{O}_{(z_0,0)}$  denotes   the holomorphic functions in a neighborhood of $(z_0,0)$ in $\C^{2 }$.
\end{lemma}

\begin{proof} 
Let $f\in H^2(\Omega,0)$. Then for each $\varepsilon >0$, there exists $j_0 = j_0 (f,\varepsilon) >0$ such that 
$$
\left  \| f- \sum_{j=0} ^ {j_0} f_j(z) w^j  \right  \| _ {L^2(\Omega,0)} < \varepsilon.
$$
Next, we use two steps to prove this theorem.\\
\emph{Step 1.}  For each $f_j \in H^2(D, 2(j+1) \varphi) $ and $\varepsilon >0$ , we try to find $F_j \in  H^2(D, 2(j+1) \varphi)$ which is holomorphic in a neighbourhood of $z_0$ in $\C$ such that 
$$
\| F_j - f_j\| _{   L^2 ( D, 2(j+1) \varphi) } < \varepsilon.
$$
Fix $z_0 \in \partial D$.  Since $z_0$ is non-isolated,  there exists a sequence  $\{z_k \}_k \subset D^c$,  $z_k \neq z_0$, $z_k \rightarrow z_0$. Put $D_k := D \cup \Delta (z_0, r_k)$,  where $r_k : = |z_k -z_0 |$.  Put $\delta_k : = \delta_{D_k} $, $ \delta : =\delta_D $ and $\Delta_{r_k} (z_0)  : =  \Delta (z_0, r_k) $.  We claim that 
\begin{eqnarray*}
\delta_k \leq 3 \delta    \ \ \ \  \text{ on }  \ \ \ \  D\cap ( \Delta _{\sqrt{2}r_k}\setminus\Delta_{2r_k}  ).
\end{eqnarray*} 
Actually,  we may assume that $r_k \searrow 0 $,  so $z_k \in \partial \Delta_{r_k} (z_0)$.   Let $z\in D \cap (  \Delta _{\sqrt{2}r_k} \setminus \Delta_{2r_k} ) $.\\
\emph{Case 1.} If $\delta(z)  < |z-z_0 | -r_k $,  then we may find $\zeta_0 \in \partial D$ satisfying 
 $\delta (z) = |z-\zeta_0| < r_k$ and 
 $$
 |z- \zeta_0| \leq |z-\zeta| ,  \ \ \ \  \text{ for every } \   \zeta \in \partial D.
 $$
On the other hand,  if $\zeta \in \partial \Delta_k$, then 
$$
|z-\zeta| \geq |z-z_0| - |z_0 - \zeta| = |z-z_0| - r_k > \delta(z) = |z-\zeta_0|.
$$
Hence  $|z-\eta | \geq |z-\zeta_0| $ for every  $\eta \in \partial D_k \subseteq \partial D \cup \partial \Delta_{r_k} $. Thus,  $\delta_k(z)  =|z-\zeta_0| = \delta(z)$.\\
\emph{Case 2.}  If $\delta(z) \geq |z-z_0| -r_k \geq r_k$,
then 
\begin{eqnarray*}
\delta_k (z) & \leq & |z-z_k| \\
& \leq & |z-z_0|  +    |z_0-z_k|  \\
& = & |z-z_0 |  + r_k\\ 
& = & |z-z_0 |  -r_k + 2r_k\\ 
& \leq &3( |z-z_0 |  - r_k) \\ 
& \leq  & 3 \delta(z). \\ 
\end{eqnarray*}
Now we fix $0 \leq j \leq j_0$. Put $\psi_k (z) = - 2\alpha  (j+1) \log  \delta_k(z)  \in C(D_k)  \cap SH(D_k)$. Choose $\chi : \R  \rightarrow [0,1]$ such that $\chi \mid _{ (-\infty, -\log 2] }= 1$,  $\chi \mid_{\R^+} = 0 .$ Set
 $$
 \eta_{j,k} : = f_j \chi ( \log ( -\log |z-z_0|) - \log ( -\log (2r_k)) ).
$$
Applying the Donnelly-Fefferman estimate to $\psi_k$, we find a solution $u_{j,k}$ of $ \overline {\partial } u = \overline{  \partial  }  \eta_{j,k}  =  f_j \overline{  \partial  }  \chi $ on $D_k \setminus \{ z_0\} $ satisfying
\begin{eqnarray*}
\int_ {D_k  \setminus \{ z_0\}} |u_k|^2 e^{-\psi_k }   & \leq &  C \int _{D_k} |f_j|^2  | \overline{\partial}  \chi(\cdot) |^2 _{ i \partial \overline{ \partial } ( - \log ( -\log |z-z_0|)) } e^{ -\psi_k} \\
& \leq & C \int _{ D \cap ( \Delta _{\sqrt{2}r_k}  \setminus  \Delta_{2r_k}  ) } |f_j| ^2 \delta _k^{2\alpha (j+1) } \\
& \leq & C  \int _{ D \cap ( \Delta _{\sqrt{2}r_k}  \setminus  \Delta_{2r_k}  ) } |f_j| ^2 \delta  ^{2\alpha (j+1) }  \rightarrow 0 \ \ ( k \rightarrow \infty).
\end{eqnarray*}
Put $$F_{j,k} = f_j \chi ( \log ( -\log |z-z_0|) - \log ( -\log (2r_k)) ) - u_{j,k}.$$
Then $F_{j,k}  \in H^2 (D, 2(j+1) \varphi) \cap \mathcal{O} (D_k)$ and 

\begin{eqnarray*}
 && \int_ D |F_{j,k} - f_j|^2 \delta  ^{2\alpha (j+1) }    \\
  &   \leq &   2 \int _{ D \cap  \Delta (z_0, \sqrt {2} r_k) } |f_j| ^2 \delta  ^{2\alpha (j+1) }  + 2 \int _{ D }   |u_{j,k}| ^2 \delta  ^{2\alpha (j+1) }  \\
& \leq & 2  \int _{ D \cap  \Delta (z_0, \sqrt {2}r_k ) } |f_j| ^2 \delta  ^{2\alpha (j+1) }   + 2 \int _{ D_k }   |u_{j,k}| ^2 \delta_k  ^{2\alpha (j+1) } \\
& \leq &  C  \int _{ D \cap  \Delta (z_0, \sqrt {2}r_k ) } |f_j| ^2 \delta  ^{2\alpha (j+1) }    \rightarrow 0  \ \ \  ( k \rightarrow \infty).
\end{eqnarray*}

\emph{Step 2.} Let $\widetilde{F} (z,w) : =  \sum ^{j_0} _ {j=0}  F_{j,k}  (z) w^j \in \mathcal {O} ( (D\cup \Delta(z_0,r_k))  \times  \C) $,  $k \gg1$, where $F_{j,k}$ is as in \emph{Step 1.}
Thus 
\begin{eqnarray*}
&& \left \| \widetilde{F} (z,w) -  \sum ^{j_0} _ {j=0}  f_j (z) w^j  \right \|^2 _ {L^2(\Omega,0) } \\
& = &  \left \|  \sum ^{j_0} _ {j=0} (  F_{j,k} (z) -  f_j (z) )w^j \right \|^2 _ {L^2(\Omega,0) } \\
& \leq  &   \frac{ \pi }{ j+1} \sum ^{j_0} _ {j=0} \int _D |   F_{j,k}   -  f_j  | ^2  \delta ^{2\alpha (j+1) }   \\
&< &   C \varepsilon.
\end{eqnarray*}
It follows that
\begin{eqnarray*}
&&  \left \| \widetilde{F}  -f  \right \|^2 _ {L^2(\Omega,0) } \\ 
& = &  \left \| \widetilde{F}  (z,w) - \sum ^{j_0} _ {j=0}  f_j (z) w^j  +   \sum ^{j_0} _ {j=0}  f_j (z) w^j -f  (z,w) \right \|^2 _ {L^2(\Omega,0) } \\
& \leq  &  2 \left \| \widetilde{F}  (z,w)  - \sum ^{j_0} _ {j=0}  f_j (z) w^j   \right \|^2 _ {L^2(\Omega,0) }  +  2 \left \|  \sum ^{j_0} _ {j=0}  f_j (z) w^j -f  (z,w) \right \|^2 _ {L^2(\Omega,0) } \\ 
&\leq &2 ( C +1)\varepsilon.
\end{eqnarray*}
\end{proof}

\begin{proof} [Proof of Theorem \ref{th:Bergmancomplete1}]
The ``if'' part follows directly from  Lemma \ref{lm:exhaustion},  Lemma \ref{lm:approxi} and Lemma \ref{lm:kobayashi}.
 For the other direction,  we prove it by contradiction.  If there is a point $z_0$ which is isolated,  then we may find a small punctured disc $D^*\subset D$ with center at $z_0$,  such that 
 $$
 \Omega \cap  (D^* \times \C )  = \{ (z,w) \in D^*  \times \C:  |w| < (\delta_{D^*} (z)) ^\alpha  = |z|^ \alpha \},
$$
which is not Bergman complete (see \cite{Chen2013}).   Thus $\Omega$ is not Bergman complete in view of the well-known localization lemma of Bergman metric.  The proof is complete.

  \end{proof}

\begin{proof}[  Proof \,of\,  Theorem \ref{th:approxidelta1}]
According to Lemma \ref{lm:exhaustion},  we know that $\Omega$ is Bergman exhaustive.  Combining Lemma \ref{lm:exhaustion}, Theorem \ref{th:deltabijin1}   and the Kobayashi criterion (Lemma \ref{lm:kobayashi}),   we obtain Theorem \ref{th:approxidelta1}.

\end{proof}

\section{Appendix: Examples for Theorem  \ref{cor:Main22} }\label{ex: for Corollary}
 
\begin{proposition}\label{prop:hyindex}
For every bounded hyperconvex domain $\Omega \subset \C^n$,  we have $\alpha(\Omega) \leq 1$.
\end{proposition}
\begin{proof}
We give a proof by contradiction. Assume $ \alpha(\Omega)> 1$.  Then there exist $\rho \in C(\Omega)  \cap PSH^-(\Omega)$,  $C>0$ and $\alpha>1$ satisfying $-\rho     \leq C\delta_\Omega^\alpha $. Fix a point $z_0\in \Omega$.  We may choose $w_0 \in \partial \Omega$ so that $|z_0 -w_0| = \delta_\Omega(z_0)$.  Applying  the Hopf lemma to the ball $B(z_0, \delta_\Omega(z_0)),$   we have a constant $c>0$ such that 
$$
\limsup _{z \rightarrow w_0} \frac{\rho(z)}{|z -w_0|} \leq -c ,
$$
when $z$ is restricted to the line segment from $z_0$ to $w_0$,  which implies $-\rho(z) \geq c \delta_\Omega (z)$.  This contradicts the initial assumption.
\end{proof}

\begin{proposition}\label{prop:index}
Let $D \subset \R^n$ be a bounded domain.  Set 
$$
\Omega_k : = \{z = x+iy \in D + i \R^n \subset \C^n: k|y|^2 < \delta _D(x)^2    \} ,
 $$
for $k>0$ where 
$$
\delta_D(x) = 
\begin{cases} 
 dist  (x, \partial D) \  &    x \in D; \\
-   dist  (x, \partial D) \  &    x \in D^c. \\
\end{cases}
$$
Then we have  $\alpha(\Omega_k) =1$ if $k\geq 1$.
\end{proposition}

\begin{remark}
Note that the boundary of $\Omega_k \cap \{ y =0 \} = D\times \{ 0\}$ may be highly irregular!
\end{remark}

\begin{proof} 
Let $\rho (z) : = k|y|^2-  \delta _D(x)^2$,  $z = x+iy$.  Obviously, $\rho$ is a Lipschitz  function and $\rho|_{\partial \Omega_k} =0$.  Thus, we may  find a constant $C_k>0$,  such that
$$
-\rho (z) \leq C_k \delta_{\Omega_k}(z).
$$
Combining this with Proposition \ref{prop:hyindex}, we only need to verify that $\rho \in PSH (\Omega_k)$.  Given $a\in \R^n,$ define
$$
u_a(x) : = -|x-a|^2 + |x|^2 = |a|^2 - 2 \sum _{j=1} ^n a_j x_j.
$$
Since $u_a(x)$ is a convex function and $\delta_D(x) = \min _{a\in \partial D} \{|x -a| \}$,  it follows that
$$
- \delta_D(x)^2 + |x|^2 = \max _{a\in \partial D} \{  u_a(x)   \}
$$
is also convex. Let $h(z) :=  |y|^2 - |x|^2, z= x+iy $.  Note that
\begin{eqnarray*}
\frac{\partial ^2 h(z) }{ \partial z_j \partial \overline{z}_k} &=& \frac{\partial ^2 |y|^2  }{ \partial z_j \partial \overline{z}_k}  - \frac{\partial ^2 |x|^2 }{ \partial z_j \partial \overline{z}_k}   \\ 
& =& \frac1 4 \left ( \frac{\partial ^2 |y|^2  }{ \partial y_j \partial  y _k}  - \frac{\partial ^2 |x|^2 }{ \partial x_j \partial  x _k}  \right )  \\
& =& 0.\\
\end{eqnarray*}
Thus $h(z)$ is psh on $\C^n.$ Hence 
$$
\rho(z) = (k-1) |y|^2 + h(z) + (  |x|^2- \delta_D(x)^2)
$$
is psh on $\Omega_k$ for $k \geq 1$.
\end{proof}

\begin{proposition}
Let 
$$
D_r^* : =  \{(x_1, x_2)\in \R^2  \mid 0 < x_1 ^2 + x_2^2 <r^2 <1\}.
$$ 
Then 
$$
\Omega_k : = \{z = x+iy \in D_r^*  +i  \R^2 \subset \C^2: k|y|^2 < \delta _{D_r^* } (x)^2    \} 
 $$
 is pseudoconvex if and only if $ k \geq 1$.
 
\end{proposition}
\begin{proof}
Let $\rho (z) : = k|y|^2-  \delta _{D_r^* } (x)^2$,  $z = x+iy$.  Then  we have $\delta^2_{D_r^* }  (x) =  x_1 ^2 + x_2^2 $ near $0$.
 Set 
 \begin{eqnarray*}
 \rho (z_1, z_2) : & = &  k (y_1^2 + y_2^2) - (  x_1 ^2 + x_2^2  ) \\  \nonumber
 & = & - \frac{ k + 1 }{ 4} ( z_1 ^2 + \overline{z}_1 ^2 + z_2 ^2 + \overline{z} _2^2  ) +  \frac{k-1}{2}(|z_1|^2  +  |z_2|^2) .  \\  \nonumber
 \end{eqnarray*}
So
  $$
  i \partial  \overline{ \partial} \rho = \frac{k-1}{2}( id z_1  \wedge d \overline{z}_1 +  id z_2 \wedge d \overline{z}_2)<0.
  $$ 
Thus,  $\Omega_k$ is not pseudoconvex at $0$ if $ k <1 $.
 
  \end{proof}
   

To provide some examples of Theorem \ref{cor:Main22},  we now consider a class of Zalcman-type domain $D\subset \R^2.$ 
Let 
\begin{eqnarray} \label{eq:construD}
D= \D \setminus \overline{  \bigcup _{l=1} ^\infty D(x_l,r_l)} 
\end{eqnarray}
with $0< r_l \ll x_l$, where $D(x_l,r_l)$ denotes the disc with center at $x_l$ and radius $r_l$.
We may assume $r_l =  \frac{1}{2^{3l}} $ and  $x_l =\frac{1}{2^l}$ for every $l\geq 1$.
Let  $  t_j : = \frac{1}{2^{j }} $ for each $j \geq 1$.   Define
\begin{eqnarray}\label{eq:eq:foroj}
D^{t_j}= D\left(0,1+\frac{1}{ 2 ^{2 ^{(j/3)}}}\right) \setminus\overline{ \bigcup _{l=1} ^j  D(x_l,\widetilde{ r_l})}
\end{eqnarray}
with $r_l - \widetilde{ r_l} = \frac{1}{ 2 ^{2 ^{(j/3)}}}$ for every $1\leq l \leq j$.
\begin{eqnarray}\label{eq:Dforsup}
\Lambda_j(D):= \sup _{ x \in \partial D}  \delta_{D^{t_j}} (x) &= & x_j - \widetilde{r_j}\\ \nonumber 
  &= &  \frac{1}{2^{j }} -  \frac{1}{2^{3j }}   +  \frac{1}{ 2 ^{2 ^{(j/3)}}}  \\   \nonumber 
  &\leq &  \frac{1}{2^{j }} -  \frac{1}{2^{3j }}   +  \frac{1}{2^{3j }} , \ \    j \gg 1   \\    \nonumber 
    & = &   \frac{1}{2^{j }}   \rightarrow 0,  \ \ \  j\rightarrow \infty. \nonumber 
\end{eqnarray}
On the other hand, 
\begin{eqnarray}\label{eq:Dforinf}
\lambda_j (D): = \inf _{x \in \partial D}  \delta_{D^{t_j} } (x)  &=&    r_j - \widetilde{r_j} \\ \nonumber 
&=&    \frac{1}{ 2 ^{2 ^{(j/3)}}} \\ \nonumber 
  & =   & e^{- \frac{\log 2}{ \left (1/2^j \right)^{1/3} }} : = e^{ {- \frac{\log 2}{  (t_j  )^{1/3} }}}.\\ \nonumber 
\end{eqnarray}
Combining the formula (\ref{eq:Dforsup} )and  (\ref{eq:Dforinf}),  we obtain 
$$
  e^{ {- \frac{\log 2}{  (t_j  )^{1/3} }} } \leq  \delta_{D^{t_j} } (z) \leq t_j,    \ \  j \gg 1
$$
for every $x \in \partial D$.

Now let $D$ and $D^{t_j} $ be as in (\ref{eq:construD}) and (\ref{eq:eq:foroj}) respectively.   Set
$$
\Omega_k = \{z = x+iy \in D+ i\R^2 \subset \C^2: k|y|^2 < \delta _D(x)^2    \}  ,  \  k \geq 1
$$
and
$$
\Omega_{k } ^{t_j } = \{z = x+iy \in D^{t_j} + i\R^2 \subset \C^2: k|y|^2 < \delta _{D^{t_j}} (x)^2    \},    \  k \geq 1 .
$$
We claim that $\Omega_k $  and $ \{ \Omega_{k } ^{t_j }  \} _j$ satisfy Theorem \ref{cor:Main22}.
That is,  there exists some constant $C$ satisfying
$$
  C e^{ {- \frac{\log 2}{  (t_j  )^{1/3} }} } \leq  \delta_{\Omega_{k } ^{t_j }  } (z) \leq t_j,    \ \  j \gg 1
$$
for every $z\in \partial \Omega_k $.  Here $\beta = \frac 1 3 < \frac{\alpha(\Omega_k)}{2} = \frac {1}{2}$ in view of the Proposition \ref{prop:index}.

Put $ \Lambda _j  (\Omega_k) : =\sup _{z \in \partial  \Omega_k  }  \delta_{\Omega_{k } ^{t_j } } (z) $ and $  \lambda _j (\Omega_k): =\inf _{z \in \partial  \Omega_k }  \delta_{\Omega_{k } ^{t_j } } (z)$.   It is sufficient to prove
\begin{eqnarray}\label{eq:con1}
 \Lambda _j  (\Omega_k)  \leq  \Lambda _j  (D)  \leq  t_j,   \ \ \  j \gg 1 
 \end{eqnarray}
 and 
\begin{eqnarray}\label{eq:con2}
  \lambda_j  (\Omega_k)  \geq  \frac{1}{2 \sqrt{k}}e^{ - \frac{\log 2}{  ( t_j   )^{1/3 }  }  } ,   \ \ \  j \gg 1 .
 \end{eqnarray}
Let $z= x+iy\in \partial  \Omega_k $. Then either $x\in D, $  $|y| = \frac{ 1}{\sqrt{k} } \delta_D(x)$ or $x\in \partial D, y=0$.  In both cases, we have $|y| =  \frac{ 1}{\sqrt{k} } \delta_D(x)$.  We first prove (\ref{eq:con1}).  Let $x^* \in \partial D $ with $|x-x^*| = \delta_D(x)$.  Then 
$$
\delta_D(x) \leq \delta_{D^{t_j}} (x) \leq |x-x^*| +  \delta_{D^{t_j}} (x^*) \leq \delta_D(x)  + \Lambda_j(D).
$$ 
 Note that,  for $y^*$ with $|y^*|= \frac{ 1}{\sqrt{k} } \delta_{D^{t_j}} (x)$ we have $z^* = x+iy^* \in \partial \Omega_{k } ^{t_j } $. We may choose $ y^* = \frac{ \delta_{D^{t_j}} (x) }{ \delta_D(x) } y$.  Then 
 \begin{eqnarray*}
\delta_{\Omega_{k } ^{t_j } } (z) & \leq & |z-z^*| \\
&=& |y-y^*|\\
&=&  |y| \cdot \left  |1- \frac{ \delta_{D^{t_j}} (x) }{ \delta_D(x) } \right |   \\
& = &  \frac{1}{ \sqrt{k} } |   \delta_D(x) -   \delta_{D^{t_j}} (x) | \\
& \leq &\frac{  \Lambda_j(D)}{ \sqrt{k} } \\
& \leq &\frac{ t_j}{ \sqrt{k} } \\
& \leq &   t_j  , \ \ \  \text{for every }  \ \ z\in  \partial \Omega_k .
 \end{eqnarray*}
 Hence,  we have 
 \begin{eqnarray*} \label{eq:for1}
 \Lambda_j  (\Omega_k) =  \sup _{z \in \partial \Omega_k }  \delta_{\Omega_{k } ^{t_j } } (z)   \leq  t_j   .
 \end{eqnarray*}
 Next we prove (\ref{eq:con2}). It is sufficient to prove that, for every $z \in \partial \Omega_k$, $z'\in \C^2$ satisfying $|z'-z| < \frac{ \lambda_j(D)}{2} $ implies $z'\in \Omega_{k } ^{t_j } $. Furthermore, it is sufficient to prove that, for every $z=x+iy \in \partial \Omega_k$, $z' = x'+i y' \in \C^2$ satisfying $|x'-x| <\frac{ \lambda_j(D)}{2}$ and $|y'-y| <\frac{  \lambda_j(D)}{ 2 \sqrt{k} } $ implies $z' = x' + iy' \in \Omega_{k} ^{t_j} $.
 We first let $x_j^* \in \partial D^{t_j}$ such that $|x-x_j^* |  = \delta_{D^{t_j}} (x)$. Consider the line segment $l$ connecting $x$ and $x_j^*$, there is a point $x^{**}  $ in $l$ such that $x^{**} \in \partial D$. Thus
 \begin{eqnarray*}
 \delta_{D^{t_j}} (x) &=& |x-x_j^*|\\
  &  =&   |x-x^{**}|  + |x^{**} -x_j^*| \\
 &  \geq  & \delta_D(x) +  \delta_{D^{t_j}}(x^{**})  \\
 &  \geq  &   \delta_D(x)  +  \lambda_j  (D)   .
 \end{eqnarray*}
Then if $|x'-x| < \frac{1}{2}   \lambda_j  (D) $ we have  $x' \in D^{t_j}$.  Furthermore,  we have 
  \begin{eqnarray*}
 \delta_{D^{t_j}} ( x' ) & \geq & \delta_{D^{t_j}} (x) -   |x-x' | \\
 &  \geq  & \delta_D(x) +   \lambda_j  (D)  -r  \\
 & >  &   \delta_D(x)  + \frac{   \lambda_j  (D) }{2}.
 \end{eqnarray*}
 Thus, if $|y'-y| < \frac{  \lambda_j(D)}{ 2 \sqrt{k} }$, then $ |y'|  < \frac{1} {\sqrt{k}}  ( \delta_D(x)  + \frac{ \lambda_j(D)}{2}) = |y| + \frac{  \lambda_j(D)}{ 2 \sqrt{k} }<  \frac{1} {\sqrt{k}}   \delta_{D^{t_j}}(x') $. That is, $z= x'+ i y' \in \Omega_{k } ^{t_j } $.    So
 \begin{eqnarray*} \label{eq:for2}
  \lambda_j(\Omega_k ) \geq \frac{  \lambda_j(D)}{ 2 \sqrt{k} } \geq  \frac{ 1}{ 2 \sqrt{k} }   e^{ {- \frac{\log 2}{  (t_j  )^{1/3} }}} ,    \ \  j \gg 1.
 \end{eqnarray*}
 Hence 
 $$
    \frac{ 1}{ 2 \sqrt{k} }  e^{ {- \frac{\log 2}{  (t_j  )^{1/3} }} } \leq  \delta_{\Omega_{k } ^{t_j }  } (z) \leq t_j,    \ \  j \gg 1
$$
for every $z\in \Omega_k $.

\textbf{Acknowledgements} The third author would like to give thanks to Dr. Yuanpu Xiong for his many valuable discussions  and also to Dr. S$\rm{\acute{e}}$verine Biard for her valuable comments during the writing of the paper.

\end{document}